\documentclass[12pt,a4paper, reqno]{amsart}
\usepackage{amsmath,amsthm,amssymb,amsfonts,xspace,tensor}
\usepackage[bookmarksnumbered,colorlinks]{hyperref}
\usepackage[T1]{fontenc}
\usepackage[utf8]{inputenc}
\usepackage[english]{babel}
\usepackage{amsxtra}
\usepackage{enumitem}
\usepackage{verbatim}
\usepackage[normalem]{ulem}
\usepackage{tikz, tikz-cd}
\usepackage[mathscr]{eucal}
\usepackage{fourier}

\newcommand{\N}{\mathbb N}

\newcommand{\R}{\mathbb R}

\newcommand{\D}{\ensuremath{\mathcal{D}}\xspace}

\renewcommand{\H}{\mathcal{H}}

\newcommand{\A}{\ensuremath{\mathcal{A}}\xspace}
\newcommand{\Lst}{\Lambda^p_{\mathrm{st}}}

\newcommand{\Ost}{\Omega^{p, \mathrm{st}}}

\newtheorem{theorem}{Theorem}[section]
\newtheorem{lem}[theorem]{Lemma}

\newtheorem{prop}[theorem]{Proposition}
\theoremstyle{definition}
\newtheorem{definition}[theorem]{Definition}
\newtheorem{remark}[theorem]{Remark}

\theoremstyle{definition}

\newcommand{\bt}{\begin{theorem}}                     
	\newcommand{\et}{\end{theorem}}                       
\newcommand{\bd}{\begin{definition}}                  
	\newcommand{\ed}{\end{definition}}                    
\newcommand{\bl}{\begin{lem}}                       
	\newcommand{\el}{\end{lem}}                                   
\newcommand{\bpr}{\begin{prop}}                  
	\newcommand{\epr}{\end{prop}}                    
\newcommand{\bere}{\begin{remark}}                      
	\newcommand{\ere}{\end{remark}}                         
\newcommand{\beq}{\begin{equation}}
	\newcommand{\eeq}{\end{equation}}
\def\bal#1\eal{\begin{align}#1\end{align}}              
\def\baln#1\ealn{\begin{align*}#1\end{align*}}          
\def\bml#1\eml{\begin{multline}#1\end{multline}}        
\def\bmln#1\emln{\begin{multline*}#1\end{multline*}}  
\def\bga#1\ega{\begin{gather}#1\end{gather}}
\def\bgan#1\egan{\begin{gather*}#1\end{gather*}}

\begin{document}

	\title[On homotopy properties of solutions of some differential inclusions]{On homotopy properties of solutions of some differential inclusions in the $W^{1,p}$-topology}
	
	\author[E. Caponio]{Erasmo Caponio}
	\address{Dipartimento di Meccanica, Matematica  e Management, Politecnico di Bari,  Bari, Italy}
	\email{erasmo.caponio@poliba.it}
	
	\author[A. Masiello]{Antonio Masiello}
	\address{Dipartimento di Meccanica, Matematica  e Management, Politecnico di Bari, Bari, Italy}
	\email{antonio.masiello@poliba.it}
	\thanks{E.C. and A.M are partially supported by European Union - Next Generation EU - PRIN 2022 PNRR {\it ``P2022YFAJH Linear and Nonlinear PDE's: New directions and Application''}, and GNAMPA INdAM – Italian National Institute of High Mathematics. This work was also partially supported by the Italian Ministry of University and Research under the Programme 
		“Department of Excellence” Legge 232/2016 (Grant No. CUP - D93C23000100001). }

	\author[S. Suhr]{Stefan Suhr}
	\thanks{S.S. is partially supported 
		by the Deutsche Forschungsgemeinschaft (DFG, German Research Foundation) -- Project-ID 281071066 -- TRR 191.}
	\address{Fakult\"at f\"ur Mathematik, Ruhr-Universit\"at Bochum, Bochum, Germany} 
	\email{Stefan.Suhr@rub.de}
	

	

	\begin{abstract}
		We consider  a differential inclusion on a manifold, defined by a field of open half-spaces whose boundary in each tangent space is the kernel of a one-form $\omega$. We make the assumption  that the corank one distribution associated to the kernel of $\omega $ is  completely nonholonomic of step $2$.  We identify a subset of solutions of the differential inclusion, satisfying   two endpoints  and  periodic boundary conditions, which are homotopy
		 equivalent in the $W^{1,p}$-topology, for any $p\in [1,+\infty)$, to  the based loop space and the free loop space respectively.  
	\end{abstract}
	\maketitle
	
	\section{Introduction}
	Let $M$ be a connected, smooth manifold and let us consider the  affine control system on $M$
	\beq\label{affine}\dot \gamma=X_0(\gamma)+\sum_{i=1}^{d}u_iX_i(\gamma),\eeq
	where $X_0, X_i\in \Gamma(TM)$, $i=1,\ldots d$, are smooth vector fields on $M$ such that the set $\{X_i\}_{i\in \{1,\ldots, d\}}$ generates a distribution $\D$ of constant rank.  The set of solutions $\gamma\colon I\subset \R \to M$ of \eqref{affine}  can be endowed with the Sobolev topology $W^{1,p}$ by taking controls  $u_i$, for all $i\in\{1,\ldots,d\}$, in the space $L^p(I,\R)$.
	
	Let us assume that $\{X_i\}_{i\in \{1,\ldots, d\}}$  satisfy the 
	H\"ormander condition, i.e.  for each $x\in M$, a finite number of their iterated brackets span the whole tangent space $T_xM$  (we say also that $\D$ is  completely nonholonomic ).  
	In a recent paper \cite{BoaLer17}, the authors prove that the set of solutions of \eqref{affine} connecting two given points in $M$ is homotopically equivalent in the $W^{1,p}$-topology to the based loop space of $M$.   
	If the drift $X_0$ is $0$ then $p$ can be any number in $[1,+\infty)$.  In this case, the same result    holds for periodic solutions and the free loop space as shown in  \cite{LerMon19}.
	If the drift does not vanish, the value of $p$ in \cite{BoaLer17} is confined to  an interval  $[1, p_c)$ where $p_c$ is at most $\sigma/(\sigma -1)$, with  $\sigma$ equal to  the step of the distribution, 
	i.e. the minimal number plus 1 of Lie bracket iterations necessary to generate $\Gamma(TM)$ from the base fields $\{X_i\}_{i\in \{1,\ldots, d\}}$. 
	
	It is then natural to ask if this last result can be improved by considering a variable drift, for example by multiplying $X_0$ with a  positive control $u_0$. 
	We are able to answer positively to this question 
	 at least when the distribution associated to $\{X_i\}_{i\in \{1,\ldots, d\}}$ is generated by the kernel of a smooth, nowhere vanishing,  one-form $\omega$ on $M$, provided that it is of step $2$,  and $X_0$ is transversal to it.
	The control problem  we consider is   equivalent to the  differential inclusion 
		\beq\label{diffincl}
	\dot\gamma\in \H,
	\eeq
	 where $\H$ is the field of open half-spaces $\{v\in TM:\omega(v)<0\}$ with the union of the zero section of $TM$. Actually,  the 
	trajectories   we consider are the ones associated with some specific   controls $u:[0,1]\to \R^{d+1}$, $u(t)=\big(u_0(t), u_1(t),\ldots, u_d(t)\big)$ that constitute a subset  \A   of $L^\infty([0,1], \R^{d+1})$. More specifically, let $D$ be the set of partition of the interval $[0,1]$, then
	\bmln \A:=\Big\{u\in L^\infty\big([0,1], \R^{d+1}\big)\big|\, u_0\geq 0,\ \exists P\in D\colon\\  \forall J\in P,\  u_0|_J=\mathrm{const.}=\xi_J 
		\text{ and,  $\forall i\in\{1,\ldots, d\}$,  $u_i|_J=\xi_j\alpha_{Ji}$}\Big\}.\emln
	The functions $\alpha_{J}:J\to \R^d$, $\alpha_J(t)=\big(\alpha_{J1}(t), \ldots, \alpha_{Jd}(t)\big)$
	have images in the closed ball of $\R^d$ centred at $0$ and having radius $K$, where $K$ is   
	 depends on  $m=\mathrm{dim} M$,  on the one-form  $\omega$,  and on a sequence  of local frame fields for $\Gamma(\D)$  defined  on neighbourhoods of $M$  that   cover $M$   (see Remark~\ref{Yi}, assumptions~\ref{ass1}--\ref{ass2} and  \eqref{K}).     The type of controls considered  naturally align  with the construction of a  local {\it cross-section }  $\sigma$ for the endpoint map (see Proposition~\ref{propBL2}). Anyway, the main reason that let led us to consider the set $\A$ is   that  this class of  controls   enables us to  obtain  energy bounds on compact subsets (in the $H^1$-topology) of the   trajectories  space for  certain singular Finsler metrics in   \cite{CaJaMa22cor} that generalize Kropina metrics (see, e.g.,  \cite{CaGiMS21}).
	 
	 We emphasize that our main results, Theorems~\ref{paths} and \ref{freeloop}, remain valid for a compact manifold $M$,  without assuming \ref{ass1}--\ref{ass2}, as follows: 
	 
	 \smallskip
	 \noindent{\bf Theorem.} {\it Let $M$ be a connected, compact manifold and $\D\subset TM$  a corank one, completely nonholonomic  distribution of step $2$. Then the sets $\Omega^p_{x,y}$ and $\Lambda^p$ are  homotopy equivalent    to $\Ost_{x,y}$ and $\Lst$ respectively;}
	 
	 \smallskip 
	 
	 \noindent where,   $\Omega^p_{x,y}$ and $\Lambda^p$ are the subsets of controls in $\A$ corresponding, respectively,   to the solutions of \eqref{diffincl} connecting the points $x,y\in M$ and the ones having equal initial and final point,   endowed with the $L^p$-topology, $1\leq p<\infty$, and  $\Ost_{x,y}$ and $\Lst$ are  respectively the manifold of the paths in $M$ between $x$ and $y$ and the one of the free loops in $M$,  both  endowed with the $W^{1,p}$-topology.  
	 

	\section{Differential inclusions  and Hurewicz fibrations}
	Let $(M, g_0)$ be a   connected, complete  Riemannian manifold  of dimension $m$. 
	We consider a smooth, nowhere vanishing,  one-form $\omega\in \Gamma(T^*M)$ and its  kernel distribution
	\[
	 \D:=\{v\in TM:\omega(v)=0\}. 
	\] 
	We assume that \D is    completely nonholonomic of step 2. It is not difficult to see  that this property  is equivalent to the form  $\omega\wedge d\omega$  being nowhere vanishing.
	
	The pair $(\mathcal{D},i)$ defines a {\it sub-Riemannian structure} in the sense of \cite[Definition 3.2]{AgBaBo20}, where $i\colon \mathcal{D} \hookrightarrow TM$ denotes the inclusion. 
	The Riemannian structure on $(\mathcal{D},i)$ is $g_0|_{\mathcal{D}\times\mathcal{D}}$.
	According to \cite[Corollary 3.27]{AgBaBo20} every sub-Riemannian structure is equivalent to a {\it free} sub-Riemannian structure $(\mathbf{U},f)$, i.e. a trivial vector bundle 
	$\mathbf{U}\to  M $  and a vector bundle morphism $f\colon \mathbf U\to TM$   such that there exists a surjective vector bundle morphism $\mathbf{p}\colon \mathbf{U}\to \mathcal{D}$ with $i\circ \mathbf{p}=f$ and the 
	Riemannian metric on $\mathbf{U}$ satisfies 
	\beq\label{min}
	|v|_0=\min\{|u|: u\in f^{-1}(v)\}
	\eeq
	for all $v\in \mathcal{D}$
	(here $|\cdot|_0$ and $|\cdot|$ are the norms associated, respectively, with the Riemannian metric $g_0$ and  the bundle metric on $\mathbf U$).
	
	Let $\mathcal{F}:=\{X_1,\ldots X_d\}\subset \Gamma(\mathbf{U})$ be a global orthonormal frame field on $\mathbf{U}$. Then, from \eqref{min}, for each $v\in \mathcal{D}$, it follows   
	\begin{equation}\label{E1}
		|v|_0=\min\left\{\left(\sum_{i=1}^d u_i^2\right)^{1/2}\left|  \sum_{i=1}^d u_i X_i\in f^{-1}(v) \right.\right\}.
	\end{equation}

	Let $X_0\in \Gamma(TM)$ be the smooth unit vector field on $M$ which is orthogonal to  $\D$, i.e. $g_0(X_0, X_0)=1$ and $g_0(X_0, v)=0$, for all $v \in \D$, and moreover  $-\omega(X_0)=\|\omega\|$, 
	where $\|\omega\|$ denotes  the function on $M$ given by $x\in M\mapsto \|\omega_x\|$, being $\|\omega_x\|$ the norm of $\omega_x$ w.r.t. $g_0$.
	
	\bere\label{Yi}
	For each  $z\in M$ 
	we can consider  a neighbourhood  $U_z$ of $z$ and 
	 a frame field $\{Y_i\}_{i\in \{1,\ldots, m-1\}}$ for $\Gamma(\D)|_{U_z}$  such that 
		$|(Y_i)_x|_0=1$ for all $x\in U_{z}$ and for all $i\in\{1,\ldots m-1\}$. 
		Being \D  completely nonholonomic of step $2$, up to 
	restricting $U_z$, there exist $Y_j, Y_l$ in the basis, such that $[Y_j, Y_l]$ is transversal to \D at each point in  $U_z$ and moreover   
	\beq\label{lambday}
	\lambda_z:= \inf\big(-\omega([Y_j,Y_l])\big)=\inf d\omega(Y_j,Y_l)>0.\eeq
	By rearranging the vector fields in the local frame, we can assume that the vector field $Y_j, Y_l$ are the first two $Y_1, Y_2$.
	\ere
		
We assume that 
\begin{enumerate}[label=(A\arabic*), leftmargin=1cm, series=l_after]
	\item\label{ass1} \beq\label{Omega}\Omega:=\sup\|\omega\|<+\infty;\eeq
	\item\label{ass2}  there exists  a  countable  covering    $U_{z_k}$, $k\in \mathbb N$,  of  $M$ made  by neighbourhoods $U_{z_h}$ as in Remark~\ref{Yi} and such that  	\beq\label{lambda}
	\lambda:=\inf_{k\in \N } \lambda_{z_k}>0.
	\eeq		 
\end{enumerate} 	
We notice that assumptions~\ref{ass1}--\ref{ass2} are satisfied if $M$ is compact.

	Let 
us   multiply   the vectors $Y_i$, $i\in\{1,\ldots,m-1\}$,  by the same factor
	\beq\label{K}
	K:=\big(5(m+3)\Omega/\lambda\big)^{1/2}.
	\eeq
	Let us denote these rescaled vector fields still with $Y_i$. Hence,  for the rescaled vector fields $Y_i$ the number defined in \eqref{lambda}, still denoted with $\lambda$, satisfies: 
	\beq\label{infmaxomega} \lambda >4(m+3)\Omega,\eeq
	moreover
	\beq |(Y_i)_x|_0= K, \quad x\in U_{z_k},\ k\in \N, \label{k1k2}\eeq
	for all $i\in\{1,\ldots m-1\}$.
		We consider 
	\[\H:= ([X_0]^++\D)\cup 0,\]
		where $[X_0]^+:=\{aX_0: a\in(0,+\infty)\}$ and $0$ denotes the zero section of $TM$.

	\bd
	Let  $I\subset \R$ be an interval. An absolutely continuous curve $\gamma\colon  I \to M$ is a {\it solution of the differential inclusion} \eqref{diffincl} if $\dot\gamma(t)\in \H$, a.e. on $I$.
	\ed
	We are interested in solutions which also belong  to  the Sobolev spaces $W^{1,p}(I,M)$, for  $p\in (1,+\infty)$  for a compact interval $I$. We recall that we can see $M$ as isometrically embedded in a Euclidean space $\R^N$ by Nash's isometric embedding theorem. Consequently,  $W^{1,p}(I,M)$ can be defined as the space $\{\gamma\in W^{1,p}(I,\R^N):\gamma(I)\subset M\}$, since $W^{1,p}(I,\R^N)$ 
	canonically embeds   into $C^{0,1-1/p}(I,\R^N)$.
	
	We notice that  \eqref{diffincl} is invariant by orientation preserving  reparametrisations,   thus we can assume that  its  solutions  are defined on the interval $I=[0,1]$.   
	
Since for each $i\in \{1,\ldots, d\}$,  $|X_i|=1$, by  \eqref{E1} we have $|f\circ X_i|\leq 1$. We then deduce easily that an absolutely continuous  curve $\gamma:I\to M$, with $\dot\gamma(s)\neq 0$ a.e.,   solves  \eqref{diffincl} and belongs to $W^{1,p}(I,M)$   if and only if there exist $L^p$-functions $(u_0, u_1,\ldots,u_d)=:u \colon I  \to \R^{d+1}$ (called {\it controls})  with $0<u_0=g_0(\dot\gamma, X_0)$ a.e. such that  
	\beq\label{horizontal}
	\dot{\gamma}=u_0X_0\circ\gamma +\sum_{i=1}^d u_i f\circ X_i\circ \gamma. 
	\eeq
	In particular,  since the isometric embedding can be taken closed (see \cite{Muller}), if the controls are defined on $I$,  so it is the curve $\gamma$. Vice versa,  if $\gamma\in W^{1,p}(I,M)$ solves \eqref{diffincl}, a control $u:=(u_0, u_1, \ldots, u_d)\in L^p(I, \R^d)$ exists, and it is given by $u_0=g_0(\dot\gamma, X_0)$ and  $u_\D:=\big(u_1,\ldots, u_d\big)$ equal  to the {\em minimal control} for $\dot\gamma-g_0(\dot\gamma, X_0\circ\gamma)X_0\circ\gamma$, i.e.  $\big|\dot\gamma(t)-u_0(t)X_0(\gamma(t))\big|_0=|u_\D(t)|$,  for  a.e. $t\in I$.
	
	 Let us introduce the subset  \A of controls that we consider.  
	\begin{definition}\label{setcontrols}
		Let $D$ be the set of partitions of the interval $I$.     
		We call a measurable curve $u\colon I\to \R^{d+1}$ {\it admissible} if there exists $P\in D$ such that for each  $J\in P$,  there exist  a constant  $\xi_J\ge 0$, and  measurable functions  $\alpha_{Ji}\colon J\to \R$, $i\in\{1,\ldots,d\}$,
		  with 
		  \beq\label{alpha}\sup_{s\in J} \sum_{i= 1}^d\alpha_{Ji}^2(s)\leq K^2\eeq
		  where $K$ satisfies \eqref{k1k2}, such that
		\beq\label{ui} u_0|_J=\xi_J^2\quad\text{and } 
		u_i|_J=\xi_J\alpha_{Ji}\eeq
		for all $i\in\{1,\ldots, d\}$.

		 We define  \A 	
		to be the set of all admissible curves endowed with the $L^p$-topology for some $p\geq 1$.
		\end{definition}
\bere
Notice that $\A$ contains the zero function  $u\equiv 0$, which corresponds to $\xi_J=0$ for all $J\in P$; moreover we notice that $u$ can assume the value zero only on a whole interval $\bar J\in P$, where $\xi_{\bar J}=0$.
\ere
\begin{definition}\label{conc}
	We define a {\it concatenation} operation $\star$ on \A: 
	for all $u, v\in \A$, let  $u\star v\in \A$
	be
	\[u\star v (s):=\begin{cases} u (2s),&s\in[0,1/2]\\
		v (2s-1),&s\in [1/2,1]
	\end{cases}\]
	
	We notice that $\star$ is continuous w.r.t. the $L^p$-topology, $1\leq p\leq \infty$.
\end{definition}

Using the frame field $\mathcal{F}$ we can define  
\[
S\colon M\times  \A  \to W^{1,p}(I,M),
\]
where $S(x, u)$ is the unique solution  $\gamma\colon I\to M$    to the Cauchy problem 
\beq\label{CP}
\begin{cases}
	\dot{\gamma}(t)= (u_0  X_0\circ\gamma)(t) +\displaystyle \sum_{i=1}^d u_i (f\circ X_i\circ \gamma)(t)\\
	\gamma(0)=x.
\end{cases}
\eeq
The following lemma and the subsequent proposition are quite standard; we include their proofs for the reader's convenience.
\bl\label{C0}
The solution operator $S\colon M\times  \A  \to W^{1,p}(I,M)\subset C^0(I,M)$ is a continuous map with respect to the  $L^p$-topology, $1\leq p\leq \infty$, in $\A$ and the  $C^0$-topology in the target space.
\el
\begin{proof}
Consider a sequence $(x_n, u_n)$ in 
$M \times \A$ such that $(x_n, u_n)\to (x, u)\in M\times \A$. 
By the Nash isometric embedding theorem, we can  
consider  $\R^N$ with the Euclidean topology as the target space of the curves $\gamma$ in \eqref{CP}. Since $x_n\to x$, we can take a relatively compact neighbourhood $U\subset \R^N$ of $x$ such that $(x_n)\subset U$.  The embedding being  closed (see \cite{Muller}), the vector fields $X_0$, $f\circ X_i$  can  be extended outside   $U\cap M$ to smooth bounded vector fields on $\R^N$  with compact support which are therefore  globally Lipschitz on  $\R^N$. Let us denote by $L_U$ the maximum of the Lipschitz constants of these vector fields. Let us consider the system \eqref{CP} with these modified vector fields  and let  us  still denote 
by $S$ the corresponding solution operator, $S:\R^N\times \A\to W^{1,p}(I,\R^N)$.
 Let then $C_U>0$ the maximum of the $C_0$-norms  of these modified vector fields.
We have, for all $t$ in the compact interval $I$:
\bmln
	|S(x_n, u_n)(t) -S(x,u)(t)|\leq |x_n -x|+C_U\int_0^t|(u_0)_n-u_0|ds\\+ L_U\int_0^t|u_0||S(x_n, u_n) -S(x,u)|ds \\+ C_U \sum_{i=1}^d \int_0^t |(u_i)_n-u_i|d s+ L_U \sum_{i=1}^d \int_0^t |u_i||S(x_n, u_n) -S(x, u)|d s.
	\emln
	 By the Gronwall inequality we then get
	\[|S(x_n, u_n)(t) -S(x, u)(t)|\leq b_n(t) e^{\int_0^t a(s) ds},\]
	where
	\baln
	&b_n(t)=|x_n -x|+ C_U \big(\int_0^t\big(|(u_0)_n-u_0|+\sum_{i=1}^d  |(u_i)_n-u_i|\big)d s\big )\\
	&a(t)= L_U\sum_{i=0}^d|u_i(t)|.
	\ealn
	Since 
	\[
	\int_I\big(|(u_0)_n-u_0|+\sum_{i=1}^d |(u_i)_n-u_i|\big)ds\to 0\] 
	and $x_n\to x$, we conclude  that $S(x_n, u_n)\to S(x, u)$ in the $C^0$-topology with the target space $\R^N$.  Hence there exists $\delta>0$ such that $S(x_n, u_n)$ and  $S(x, u)$ restricted to $[0,\delta]$ have image contained in $U$ and are then also solution of the initial system \eqref{CP}. We can then repeat the above argument with initial points $S(x_n, u_n)(\delta)$ and $S(x, u)(\delta)$ and a neighbourhood  of $S(x,u)(\delta)$. Being $S(x,u)$ compact in $M$ in a finite number of step we can conclude that $S(x_n, u_n)\to S(x, u)$ in the $C^0$-topology.  
	\end{proof}
	
	\bpr\label{SIC0}
$S\colon M\times  \A  \to W^{1,p}(I,M)$, $1\leq p\leq\infty$,   is a continuous map.
	\epr
	\begin{proof}
Let $(x_n, u_n)$ in $M \times \A$ such that $(x_n, u_n)\to (x, u)\in M\times \A$. From Lemma~\ref{C0}
	 the curves $S(x_n,u_n)$  are  contained  in  a compact  subset $K$   of $M$.    This allows us to control the distance associated to $g_0$ both from below and above by the Euclidean distance in $\R^N$. Hence it is enough to prove the  remaining convergence needed using the  Euclidean metric  in $\R^N$.  The vector fields $X_0|_K$, $(f\circ X_i)|_K$, can  be extended to smooth bounded vector  fields  on $\R^N$  with compact support inside an open neighbourhood of $K$ in $R^N$. These extended vector fields (that will be denoted with the same symbols $X_0$, $f\circ X_i$) are therefore  globally Lipschitz on  $\R^N$. Let us denote  by $L$ and $C$  the maximum of their Lipschitz constants and  $C^0$ norms respectively.

Let us also denote the derivatives of  the curves $S(x,u)$ and $S(x_n,u_n)$  by $\dot S(x,u)$ and  $\dot S(x_n,u_n)$, respectively. 

From the equation in \eqref{CP} and	the triangle inequality we obtain
\baln
|\dot S(x_n, u_n) -\dot S(x, u)|\le & \; |(u_0)_n  X_0\circ S(x_n, u_n) - u_0 X_0\circ S(x, u)| \\
& +\displaystyle \sum_{i=1}^d |(u_i)_n f\circ X_i\circ S(x_n, u_n)-u_i f\circ X_i\circ S(x, u)|\\
\le & \; |(u_0)_n -u_0| |X_0\circ S(x_n, u_n)|\\
& +|u_0| |X_0\circ S(x_n, u_n) -  X_0\circ S(x, u)| \\
& +\sum_{i=1}^d |(u_i)_n -u_i||f\circ X_i\circ S(x_n, u_n)|\\
& +\sum_{i=1}^d |u_i||f\circ X_i\circ S(x_n, u_n)- f\circ X_i\circ S(x, u)|.
\ealn	
Using the Lipschitz constant and the bound on the vector fields it follows 
\baln
|\dot S(x_n, u_n) -\dot S(x, u)|\le & \; C\big((u_0)_n -u_0| + \sum_{i=1}^d |(u_i)_n -u_i|\big)\\
& + L\big(|u_0| + \sum_{i=1}^d |u_i|\big)|S(x_n, u_n) -   S(x, u)|.
\ealn
Finally with Jensen's inequality we get 
\baln
|\dot S(x_n, u_n) -\dot S(x, u)|^p\le & \; (2(d+1))^{p-1} \Big[C^p\big(|(u_0)_n -u_0|^p + \sum_{i=1}^d |(u_i)_n -u_i|^p\big)\\
& + L^p\big(|u_0|^p + \sum_{i=1}^d |u_i|^p\big)|S(x_n, u_n) -   S(x, u)|^p\Big],
\ealn
from which  the  convergence of  $S(x_n, u_n)$  to $S(x, u)$ in $W^{1,p}(I,\R^N)$ and consequently in $W^{1,p}(I,M)$ easily follows.   
\end{proof}
	
Let us  define the {\it endpoint map} 
\[
F\colon  M\times \A \to M\times M, \quad (x,u) \mapsto \big(x,S(x, u)(1)\big). \]

\bere
The uniform convergence in Proposition~\ref{SIC0} especially implies that $F$ is continuous.
\ere

The space under consideration here is 
\[
\Lambda^p := F^{-1}(\triangle_M),
\]
where $\triangle_M:=\{(x,x)\in M\times M\}$ denotes the diagonal in the product $M\times M$ together with the induced topology. 

\bd \label{covering}
Let $E,B$ be topological spaces. A  continuous  map $\pi \colon E\to B$ 
is a {\it Hurewicz fibration} if it satisfies the {\it homotopy lifting property}, i.e. for every topological spaces $Z$ and every homotopy
$H\colon Z\times [0,1]\to B$ with a  continuous  lift $\tilde{H}_0\colon Z\to E$ of $H(\cdot,0)\colon Z\to B$,  relative  to  $\pi$, i.e. $\pi \circ \tilde{H}_0= H(\cdot,0)$,  there exists a  continuous  lift $\tilde{H}\colon Z\times [0,1]\to E$ of $H\colon Z\times [0,1]\to E$ relative  to  $\pi$.
\ed

The main result of this section is the following: 

\begin{theorem}\label{thmHurew}
Let  assumptions~\ref{ass1}--\ref{ass2} hold. Then  the endpoint map 
	\[
	F|_{\Lambda^p} \colon \Lambda^p \to \triangle_M
	\] 
	is a Hurewicz fibration for every $  1 \leq   p<\infty$.
\end{theorem}
Theorem~\ref{thmHurew} was proved in \cite{LerMon19} for $p=2$ and $X_0=0$;    the space of paths  between two points satisfying an affine control system of the type \eqref{horizontal}, with $u_0$ fixed equal to $1$, has been studied in \cite{DomRab12}   for $p=1$ and  in    \cite{BoaLer17} for $p>1$ ($p$ below a ``critical'' exponent related to the step of the distribution $\D$); the case of horizontal paths defined by a completely nonholonomic 
distribution  and endowed  with the uniform topology was previously considered  in \cite{Sarichev}, where credit was   given to \cite{Smale58} for certain  ideas employed. 

For the proof, we need the following proposition, which can be seen as the counterpart for the differential inclusion \eqref{diffincl} of \cite[Lemma 1]{DomRab12} and  \cite[Proposition 2]{BoaLer17} valid for an affine control system. We emphasize that in our case the upper bound on $p$, existing  for affine control systems,  is not present.
We use  a non-smooth implicit function theorem developed by F. H. Clarke. Let us introduce some related notions.

Let $A\subset \R^n$ be an open subset and $G:A\to \R^m$ be a Lipschitz map. Let  $\Omega_G\subset A$ be the set of points where $G$ is not differentiable. Let $J_G(y)$ be the Jacobian matrix of $G$ at a point $y\in A\setminus \Omega_G$. Following \cite[Definition 2.6.1]{Clarke90}, we say that the {\em generalized Jacobian} at $x\in A$,
denoted with $\partial G(x)$ is the convex hull of all the matrices obtained as limits of $J_G(x_i)$ as $x_i\to x$, $x_i\not \in \Omega_G$.  The generalized Jacobian at $x$ is said to be of {\em maximal rank} if all the matrices in  $\partial G(x)$ have maximal rank. The following theorem is  \cite[Th. 7.1.1]{Clarke90}.
\bt[Non-smooth inverse  function theorem by Clarke]\label{IFT}\mbox{}\\
Let $\partial G(x_0)$ be of maximal rank, then there exist neighborhoods $V$ and $U$ of $x_0$ and 
$G(x_0)$, respectively, and a Lipschitz function $H: U \to \R^n$ such that 
\begin{itemize}
\item[(i)] $G\big(H(u)\big) = u$ for every $u\in U$; 
\item[(ii)] $H\big(G(v)\big) = v$ for every $v\in  V$.
\end{itemize}
\et 
We use Theorem~\ref{IFT} in the proof of Lemma~\ref{lemBL2} below, where we show  local injectivity at $0\in \R^m$ of a map $G(x,\cdot):\R^m\to M$ given as the  composition of endpoint maps associated to the flows of  vector fields that are sections of $\H$ in \eqref{affine}.
Let us introduce the map $G$.

Recalling Remark~\ref{Yi}, let $h\in \N$ such that  $x_0\in  U_{z_h}$ and let   us take   a local basis $(Y_j)_{j\in \{1,\ldots, m-1\}}$ for $\Gamma(\D)$  defined in $U_{z_h}$,  and $Y_m:=[Y_1, Y_2]$. For each $j\in\{1,\ldots, m-1\}$, let $(a_{ji})_{i\in\{1,\ldots,d\}}: U_{z_h}\to \R$ such that on $U_{z_h}$ we have both
\beq \label{comb}
Y_j=\sum_{i=1}^d a_{ji} f\circ X_i,\quad\text{and}\quad \big|(Y_j)\big|_0^2= \sum_{i=1}^d a_{ji}^2
\eeq 
(recall \eqref{E1}). We notice that the vector field $a_j\colon U_{z_h}\to \R^d$, having components $\big(a_{j1}(x), \ldots, a_{jd}(x)\big)$ and which realizes  \eqref{comb}, is  smooth. In fact, for each $x\in U_{z_h}$, $a_j(x)$ is the vector in the affine subspace $f^{-1}(Y_j(x))$ of $\R^d$  which is orthogonal to the kernel of $f$ on the fibre $\mathbf U_x$. Since $f$ and $Y_j$ are both smooth, it follows  that this vector smoothly varies with $x\in U_{z_h}$. 

We   extend the vector fields $Y_j$  (after restricting them  to a smaller open set $U_{z_h}$) without changing the supremum of their $g_0$-norms,    to  obtain  global sections $Y_j$ of \D.  

For  each $j\in \{1,\ldots,m\}$,  any $\xi_j, \xi_{m_1}, \xi_{m_2}\in\R$  and any $x\in M$ let us consider 
\beq\label{Qi}
\begin{split}
	&Q_j(\xi_j)(x):=e^{\xi_j^2\tilde X_0+\xi_j \tilde Y_j}(x),\quad  \text{ for $j=1,\dots, m-1$,}\\
	&Q_m(\xi_{m_1},\xi_{m_2})(x):=\\
	&\quad \quad e^{\xi_{m_2}^2\tilde X_0+\xi_{m_2} \tilde Y_2}\circ e^{\xi_{m_1}^2\tilde X_0+\xi_{m_1} \tilde Y_1}\circ e^{\xi_{m_2}^2\tilde X_0-\xi_{m_2} \tilde Y_2}\circ e^{\xi_{m_1}^2\tilde X_0-\xi_{m_1}\tilde Y_1}(x),	
\end{split}
\eeq
being, for $X\in \Gamma(TM)$, $e^X(x)$ the value at $t=1$ of the flow of $X$ that passes through $x$ at $t=0$; the vector fields $\tilde X_0$ and $\tilde Y_j$ are respectively equal to $X_0/(m+3)$, and $Y_j/(m+3)$,  for $j\in\{1,\ldots,m-1\}$.
We define 
\bal
&G:M\times  \R^m \to M,\nonumber\\
&G(x, \xi)=Q_m(\mathrm{sgn}(\xi_m)\sqrt{|\xi_m|}, \sqrt{|\xi_m}|)\circ Q_{m-1}(\xi_{m-1})\circ\ldots\circ Q_1(\xi_1)(x)\label{G}\eal
\begin{lem}\label{lemBL2}
Let assumptions~\ref{ass1}--\ref{ass2} hold.	Then  for each $x_0\in M$,  there   exists a neighbourhood $V$ of $x_0$ and a continuous map 
	\[
		\psi \colon V\times V \to  \R^m, 
	\]
	such that  $G(x, \psi(x,y))= y$ and $\psi(x,x)=0$ for all $x,y\in V$.  
\end{lem}
\begin{proof}
	From the Baker-Campbell-Hausdorff formula (see, e.g., \cite[p.28]{Serre}) for $\xi_{m_1},\xi_{m_2}$ small enough,  we  have 
	\[Q_m(\xi_{m_1},\xi_{m_2})=e^{2(\xi_{m_1}^2+\xi_{m_2}^2) \tilde X_0 - \xi_{m_1}\xi_{m_2}[\tilde Y_1, \tilde Y_2]+\text{terms of degree $\geq 3$ in } \xi_{m_1}, \xi_{m_2}}.\]
	Thus,  $G$ 
	is Lipschitz  on a neighbourhood of $(x_0,0)$  and  $G(x, 0)=x$. We notice that actually $G$ is differentiable w.r.t. $\xi$, for each $\xi\neq 0$. 
	 Let us then consider the map 
	\[\tilde G\colon M\times\R^m\to M\times M, \quad \tilde G(x,\xi)=\big(x,G(x, \xi)\big).\] 
	It admits   generalized Jacobian  $\partial \tilde G(x,0)$ given by
	\[\begin{bmatrix}
		I & 0 \\
		\partial_x G(x,0)& A(\lambda, x,0)
	\end{bmatrix}\]
	where $A(\lambda, x,0)$ is the linear matrix pencil
	\[
	\lambda\in[0,1]\mapsto \begin{bmatrix}\tilde Y_1(x)\\ \ldots\\ \tilde Y_{m-1}(x)\\ -4\lambda \tilde X_0(x)  +4(1-\lambda)\tilde X_0(x) -  \tilde Y_m(x)\end{bmatrix}
	\]
	 We notice that $\partial \tilde G(x_0,0)$ is of maximal rank. In fact, as  both $\omega_{x_0}(\tilde X_0)$ and $\omega_{x_0}(\tilde Y_m)$ are negative (recall \eqref{lambday}), then $-4\tilde X_0(x_0)-\tilde Y_m(x_0)$ is  transversal to $\D_{x_0}$. Since  $\D_{x_0}$ is generated by the vectors in the first $m-1$ rows of $A(\lambda, x,0)$,  the maximality of the rank of  $\partial \tilde G(x_0,0)$ is ensured if  $ \omega_{x_0}(4\tilde X_0-\tilde Y_m)>0$. 
	Recalling \eqref{lambday}--\eqref{lambda}, this follows from \eqref{infmaxomega} as
	\bmln
	\omega_{x_0}(4\tilde X_0-\tilde Y_m) = - \frac{4}{m+3}\|\omega_{x_0}\|-\frac{1}{(m+3)^2}\omega_{x_0}([Y_1,Y_2])\\ > -\frac{4\Omega}{m+3}+\frac{\lambda}{(m+3)^2}.
	\emln
	 By Theorem~\ref{IFT} applied to $\tilde G$,  there exists then  a neighbourhood $V$ of $x_0$ and a Lipschitz  function $\psi\colon V\times V\to \R^m$ such that 
	\[G(x, \psi(x,y))=y, \quad \text{for all $x,y\in V$.}\]
	\end{proof}
From Lemma~\ref{lemBL2} we get the existence of a local cross-section for the endpoint map $F$.
\bpr\label{propBL2}
Let assumptions~\ref{ass1}--\ref{ass2} hold.	Then  for each $x_0\in M$,  there   exists a neighbourhood $V$ of $x_0$ and a continuous maps 
\[
	\sigma\colon V\times V \to  \A, 
\]
such that  $F(x, \sigma(x,y))= (x,y)$ and $\sigma(x,x)=0$ for all $x,y\in V$.  
\epr
\begin{proof}
	Let $\psi=(\psi_1, \ldots, \psi_m)$   the map in Lemma~\ref{lemBL2}. It defines, for each $x,y\in V$,  a vector  $\xi\in \R^m$ which  appears in  the definition of $G$ in \eqref{G}. We can  then associate  to this vector the continuous piecewise smooth curve $\tilde \gamma_{x,y}:[0,m+3]\to M$ connecting $x$ to $y$, obtained as the concatenation of the  flow lines, defined on intervals of  length  $1$ of the $m+3$ vector fields appearing in \eqref{Qi}. Since $\psi$ is Lipschitz, these
	flow lines, depending on $x$ and $y$,   uniformly converge 
	on their unit interval of definition (see the proof of \cite[Theorem 2.1, p. ~94]{Hartma64}), i.e.   the map 
	\[
	(x,y)\in V\times V\mapsto \tilde  \gamma_{x,y}\in C^0([0,m+3], M),
	\]  
where $C^0([0,m+3], M)$ is endowed with uniform convergence topology,  is continuous. We can  reparametrise $\tilde \gamma$ on the interval $[0,1]$, by taking $\gamma_{x,y}(t):=\tilde \gamma_{x,y}((m+3)t)$. Notice that $\gamma_{x,y}$ is the concatenation of the flow lines of the vector fields appearing in \eqref{Qi} with $X_0$ and $Y_j$ that replace, respectively, $\tilde X_0$, and $\tilde Y_j$; moreover  each flow line is parametrized on an interval  of  length  $1/(m+3)$. Let us then define (recall \eqref{comb}-- \eqref{G}) 
	\baln
	&(x,y)\in V\times V\mapsto  \sigma(x,y):[0,1]\to \R^{d+1}\\ 
	&\sigma(x,y) (t)=\\
	&\left\{\begin{aligned}&\big(\psi_j^2(x,y), \psi_j(x,y)  a_{j1}(\gamma_{x,y}(t)), \ldots, \psi_j(x,y)a_{jd}(\gamma_{x,y}(t))\big),\\
		& \hspace{4cm}t\in \big[(j-1)/(m+3),j/(m+3)\big),\  j\in \{1,\ldots, m-1\}\\
		&\big(|\psi_m(x,y)|, -\epsilon_m\sqrt{|\psi_m(x,y)|} a_{11}(\gamma_{x,y}(t)), \ldots, -\epsilon_m\sqrt{|\psi_m(x,y)|}a_{1d}(\gamma_{x,y}(t))\big),\\
		&\hspace{4cm} t\in \big [(m-1)/(m+3),m/(m+3)\big )\\
		&\big(|\psi_m(x,y)|, -\sqrt{|\psi_m(x,y)|} a_{21}(\gamma_{x,y}(t)), \ldots, -\sqrt{|\psi_m(x,y)|}a_{2d}(\gamma_{x,y}(t))\big),\\
		&\hspace{4cm} t\in \big [m/(m+3),(m+1)/(m+3)\big)\\
		&\big(|\psi_m(x,y)|, \epsilon_m\sqrt{|\psi_m(x,y)|} a_{11}(\gamma_{x,y}(t)), \ldots,\epsilon_m\sqrt{|\psi_m(x,y)|}a_{1d}(\gamma_{x,y}(t))\big),\\
		&\hspace{4cm} t\in \big [(m+1)/(m+3),(m+2)/(m+3)\big)\\		
		&\big(|\psi_m(x,y)|, \sqrt{|\psi_m(x,y)|} a_{21}(\gamma_{x,y}(t)), \ldots, \sqrt{|\psi_m(x,y)|}a_{2d}(\gamma_{x,y}(t))\big),\\
		&\hspace{4cm} t\in \big [(m+2)/(m+3),1\big ]
	\end{aligned}\right.
	\ealn
	where $\epsilon_m:=\mathrm{sign}(\psi_m(x,y))$.						
	Let $J_j$, $j\in \{1,\ldots, m+3\}$ be one of the intervals  of the variable $t$ in the definition above. We notice that by  construction (recall \eqref{k1k2} and \eqref{comb})
	\[
	\sup_{t\in J_j} \sum_{i=1}^d\big (a_{ji}(\gamma_{x,y}(t))\big )^2\leq K^2,
	\]
	hence   $\sigma(x,y)\in \A$. 
	Moreover,  since the map
	$(x,y)\in V\times V\mapsto \gamma_{x,y}|_J\in C^0(J,M),$ is continuous and the functions $a_{ji}$ are smooth, each  map
	\[(x,y)\in V\times V\mapsto a_{ij}\circ (\gamma_{x,y})|_J\in C^0(J,\R),\]   
	is continuous (the target space is endowed with the  topology of  uniform convergence).   Then the map 
	\[
	(x,y)\in V\times V\mapsto \sigma(x,y)\in \A
	\] 
	is continuous  when $\A$ is endowed with  the $L^p$-topology, $1\leq p\leq \infty$.
\end{proof}
We are now ready to prove Theorem~\ref{thmHurew}.

\begin{proof}[Proof of Theorem~\ref{thmHurew}]
	Since $\triangle_M\cong M$ is a metric space, by the Hurewicz uniformization theorem \cite{Hur55}, it is enough   to show that $F|_{\Lambda^p}$ is  locally a  Hurewicz fibration. 
	Thus, it suffices  to prove that the covering homotopy condition (recall Definition~\ref{covering}) holds locally, i.e. for any $x\in M$ there exists a neighbourhood  $W\subset M$, such that given  a   topological space $Z$, a homotopy $H\colon Z\times [0,1]\to \triangle_W$ 
	and a lift $\tilde{H}_0\colon Z\to  F^{-1}(\triangle_W)$ of $H(\cdot,0)$, there exists   a continuous map $\tilde{H}\colon Z\times [0,1]\to 
	F^{-1}(\triangle_W)$ with $F\circ \tilde{H}=H$ and $\tilde{H}(\cdot,0)=\tilde{H}_0$.

	Let then $x_0\in M$ be given and choose $W$ according to  Proposition \ref{propBL2}.
	
	Let $h:=\mathrm{pr_1}\circ H$, where $\mathrm{pr_1}:M\times M\to M$ is the projection on the first factor, and let $\tilde h_0:Z\to \A$ be the second component of $\tilde H_0$. Consider, for each $\zeta\in Z$ and $s\in [0,1]$, 
	\[
	\sigma\big (h(\zeta,s),h(\zeta,0)\big)\star \tilde{h}_0(\zeta)\in   \A, 
	\]
	where $\star$ is the concatenation operation in Definition~\ref{conc}.  We notice  that $S\Big(h(\zeta,s), \sigma\big(h(\zeta,s),h(\zeta,0)\big)\star \tilde{h}_0(\zeta)\Big)$ is a  path, satisfying \eqref{diffincl},   from $h(\zeta,s)$ to $\mathrm{pr_1}\big(H(\zeta,0)\big)$, i.e. 
	\[F\Big(h(\zeta,s), \sigma\big (h(\zeta,s),h(\zeta,0)\big)\star \tilde{h}_0(\zeta)\Big)=\big(h(\zeta, s), h(\zeta,0)\big).\]
	Thus, we can  complete the control $\sigma\big (h(\zeta,s),h(\zeta,0)\big)\star \tilde{h}_0(\zeta)$ to  a loop control $c(\zeta,s)$ pointed at   $h(\zeta,s)$ by taking
	\beq\label{czetas}
	c(\zeta,s):=\Big(\sigma\big (h(\zeta,s),h(\zeta,0)\big)\star \tilde{h}_0(\zeta)\Big)\star \sigma\big(h(\zeta,0),h(\zeta,s)\big).
	\eeq
	Let  $\tilde H_1:Z\times [0,1]\to \A$ be 
		\beq\label{tildeH1}
	\tilde{H_1}(\zeta,s)(t):=\begin{cases}
		c(\zeta,  s)(t/s)&t\in [0,s/4)\\
		c(\zeta, s)\big((t+1-s)/(4-3s)\big)&t\in [s/4,1-s/2]\\
		c(\zeta,s)(t/s-1/s+1)&t\in (1-s/2,1]
	\end{cases}
	\eeq
	
	Let us observe that $\tilde H_1(\zeta, 1)=c(\zeta, 1)$ and $\tilde H_1(\zeta, 0)(t)=c(\zeta,0)\big((t+1)/4\big)$. As a piecewise reparametrization  of concatenation maps by affine functions,  $\tilde H_1$ is continuous on $Z\times (0,1]$ (this can be proved  by applying the same argument used to estimate integral \eqref{2nd} below). In order to conclude the proof we need to show the continuity of $\tilde H_1$ at  $(\bar \zeta, 0)$ for all $\bar \zeta \in Z$. Indeed, defining
	\[
	\tilde H:Z\times [0,1]\to F^{-1}(\triangle_W),  \quad \tilde H(\zeta, s):=\big(h(\zeta,s), \tilde H_1(\zeta,s)\big),
	\] 
	we notice that 
	$F(\tilde H(\zeta, s))=(h(\zeta, s), h(\zeta,s))=H(\zeta, s)$;  moreover  the second component of $\tilde H(\zeta, 0)$  is equal to  $\tilde H_1(\zeta, 0)=c(\zeta, 0)\big( (t+1)/4)\big)$   which is equal to $\tilde h_0(\zeta)$ by \eqref{czetas}, hence $\tilde H(\zeta,0)=\tilde H_0(\zeta)$.

	\medskip
	\noindent {\bf Claim.}
{\it	The map $\tilde H_1:Z\times[0,1]\to \A$ in \eqref{tildeH1}  is continuous, with  \A endowed with the $L^p$-topology, $1\leq p<\infty$.}\par
\smallskip
\noindent Let $(\zeta, s)\in Z\times [0,1]$. If $s=0$ then 
 \beq\label{h0} \int_0^1|\tilde H_1(\zeta, 0)-  \tilde H_1(\bar \zeta, 0)|^pdt = \int_0^1|\tilde h_0(\zeta)-\tilde h_0(\bar \zeta)|^pdt, \eeq
  and since $\tilde h_0$ is continuous in the $L^p$ -topology of the target space, \eqref{h0} is less than any $\epsilon>0$ in a neighbourhood $U_1$ of $\bar\zeta$. Let now $s>0$ and let us consider 
 	\beq\label{intildeH1} \int_0^1|\tilde H_1( \zeta, s)-  \tilde H_1(\bar \zeta, 0)|^pdt \eeq
	 Decomposing \eqref{intildeH1}  in the sum of three integrals on the intervals $[0, s/4]$, $[s/4, 1-s/2]$ and $[1-s/2, 1]$ and changing the variable in the first and the last one, we get
	\bal  \lefteqn{ \int_0^1|\tilde H_1(\zeta, s)-  \tilde H_1(\bar \zeta, 0)|^pdt  }&\nonumber \\
	&=\int_0^{1/4}s \big|c(\zeta, s)(\tau)-c(\bar \zeta, 0)\big ((s\tau +1)/4\big )\big|^p d\tau\label{1st}\\
 &\quad  +\int_{s/4}^{1- s/2}\big|c(\zeta, s)\big( (t+1-s)/(4-3s)\big)-c(\bar \zeta, 0)\big((t+1)/4\big )\big|^p dt\label{2nd} \\
	&\quad\quad +\int_{1/2}^1 s\big|c(\zeta, s)(\tau)-c(\bar \zeta, 0)\big ((s(\tau -1)+2)/4)\big|^p d\tau\label{3rd} \eal
Let us analyse the  integral in \eqref{1st}. Since $c$ is continuous in  $(\bar \zeta, 0)$ there exists a neighbourhood $U_2\times [0,\delta)$ of $(\bar \zeta, 0)$ such that $c$ is bounded in the  $L^p$-norm and therefore, up to  decreasing   $\delta$, $\int_0^{1/4}s |c(\zeta, s)(\tau)|^p d\tau$ is less than any fixed $\epsilon>0$ on such a neighbourhood. On the other hand,
\[\int_0^{1/4}s \big|c(\bar \zeta, 0)\big ((s\tau +1)/4\big )\big|^p d\tau=4\int_{1/4}^{(s+4)/16}|c(\bar \zeta, 0)(w)|^p dw\]
which is also arbitrarily small for $s$ small enough. The integral in \eqref{3rd} can be estimated analogously. For the integral in \eqref{2nd}, we   add and subtract $c(\bar \zeta, s) \big( (t+1-s)/(4-3s)\big)$ inside the absolute value and we notice that
\bmln \int_{s/4}^{1- s/2}\big|[c(\zeta, s)-c(\bar\zeta, s)]\big( (t+1-s)/(4-3s)\big)\big|^p dt\\=(4-3s)\int_{1/4}^{1/2}\big|\big(c(\zeta, s)-c(\bar \zeta, s)\big)(\tau)\big|^p d\tau \emln
can be made  arbitrarily small in a neighbourhood of $(\bar \zeta,0)$ by the continuity of $c$  in the $L^p$-topology of the target space. It remains to show that 
\[\int_{s/4}^{1- s/2}\big|c(\bar \zeta, s)\big( (t+1-s)/(4-3s)\big)-c(\bar \zeta, 0)\big((t+1)/4\big )\big|^p dt\]is arbitrarily small for $s$ in a sufficiently small  neighbourhood of $0$.  It is enough to show that for any sequence $(s_n)$ converging to $0$, it holds  
\[\int_{s_n/4}^{1- s_n/2}\big|c(\bar\zeta, s_n)\big( (t+1-s_n)/(4-3s_n)\big)-c(\bar \zeta, 0)\big((t+1)/4\big )\big|^p dt\to 0.\]
Let  $\chi_{[1/4,1/2]}$ be the characteristic function of the interval $[1/4,1/2]$, and, for each $n\in\N$, 
\begin{align*}
 \phi_n:[0,1]\to \R,& \quad \phi_n(t) := (t+1-s_n)/(4-3s_n), \\     
 f_n:[0,1]\to [0,+\infty), & \quad f_n :=\chi_{[1/4,1/2]}\circ\phi_n\big | c(\bar\zeta, s_n)\circ\phi_n\big|^p. 
\end{align*}
Let $A\subset [0,1]$ be a measurable set. Then  we have
\baln
\int_A f_n(t) dt &= (4-3s_n)\int_{\phi_n^{-1}(A)} \chi_{[1/4,1/2]}|c(\bar\zeta, s_n)|^pd\tau\\
&\leq 2^p(4-3s_n)\int_{\phi_n^{-1}(A)} \chi_{[1/4,1/2]}|c(\bar\zeta, s_n)-c(\bar\zeta, 0)|^pd \tau\\
&\quad  +2^p(4-3s_n)\int_{\phi_n^{-1}(A)} \chi_{[1/4,1/2]}|c(\bar\zeta, 0)|^pd \tau
\ealn
Since the measures of the sets 
$\phi^{-1}_n(A)$  are  controlled from above by the measure of $A$,  we can use the absolute continuity of the last integral and the fact that $c(\bar\zeta, s_n)\to c(\bar\zeta, 0)$ in the $L^p$-norm, to conclude that the sequence $f_n$ is uniformly integrable on $[0,1]$. This implies that the sequence 
\[
\{\chi_{[1/4,1/2]}\circ\phi_n(t)\big | c(\bar\zeta, s_n)\circ\phi_n(t)-c(\bar\zeta, 0)\big((t+1)/4\big)\big|^p\}_n
\] 
is also uniformly integrable on $[0,1]$.
As  $c(\bar\zeta, s_n)\to c(\bar\zeta, 0)$ in $L^p$, there exists a subsequence, still denoted by $s_n$, such that $c(\bar\zeta, s_n)\to c(\bar\zeta, 0)$ a.e. on $[0,1]$.
Hence, for any subsequence of $(s_n)$, we have that there exists another subsequence such that 
\[
 \chi_{[1/4,1/2]}\circ\phi_n(t)\big | c(\bar\zeta, s_n)\circ\phi_n(t)-c(\bar\zeta, 0)\big((t+1)/4\big)\big|^p\to 0\;\text{ a.e. on }[0,1].
\] 
By Vitali convergence theorem (see \cite[p. 133]{Rudin}) we conclude that  
\bmln \int_{s_n/4}^{1- s_n/2}\big|c(\bar\zeta, s_n)\big( (t+1-s_n)/(4-3s_n)\big)-c(\bar \zeta, 0)\big((t+1)/4\big )\big|^p dt\\ =\int_0^1 \chi_{[1/4,1/2]}\circ\phi_n(t)\big | c(\bar\zeta, s_n)\circ\phi_n(t)-c(\bar\zeta, 0)\big((t+1)/4\big)\big|^p dt \to 0. \emln
\end{proof}
We notice that  Proposition~\ref{propBL2} gives also the following
\bpr\label{prHurew}
Let  assumptions~\ref{ass1}--\ref{ass2} hold. Then for each   $\bar x\in M$ the  endpoint map 
\[F_{\bar x} :\A \to M,\quad\quad  F_{\bar x}(u):=S(\bar x, u)(1)\]
is a Hurewicz fibration.
\epr
\begin{proof}
	As in the proof of Theorem~\ref{thmHurew}, it is enough to show that the covering homotopy condition holds locally. Using  the same notation as in the above proof, replacing $\triangle_W$ with $W$ and  $F^{-1}(\triangle_W)$ with  $F^{-1}_{\bar x}(W)$, we consider now 
	\[c(\zeta, s):=\tilde{H}_0(\zeta)\star \sigma\big (H(\zeta,0), H(\zeta,s)\big).\]
	and 
	\[
	\tilde{H}(\zeta,s)(t):=\begin{cases}
		c(\zeta, s)\big(t/(2-s)\big)&t\in [0,1-s/2]\\
		c(\zeta,s)\big((t-1+s/2)/s+1/2\big)&t\in (1-s/2,1]
	\end{cases}
	\]
\end{proof}	

\section{Homotopy equivalences of the based and the free loop space}

Let us denote by $\Lst$ the  free loop space of $M$ endowed with the $W^{1,p}$ topology,
i.e. 
\[\Lst:=\{\gamma\in W^{1,p}(I,M):\gamma(0)=\gamma(1)\}.\]
Let $x, y\in M$. We consider also the set
\[\Ost_{x,y}:=\{\gamma\in W^{1,p}(I,M):\gamma(0)=x,\ \gamma(1)=y\}\]
endowed with  the $W^{1,p}$-topology and 
\[\Omega^p_{x,y}:=F^{-1}_x(y)\subset \A.\]
Let $\Omega M$ be 
the based loop space of $M$ with base point $\bar x\in M$ endowed with the compact-open topology and  let    $i\colon W^{1,p}(I,M)\to C^0(I,M)$ be the canonical embedding.

Finally,  let 
\[\Omega_{F_{\bar x}}:=\{(u,\omega)\in \A\times M^I: F_{\bar x}(u)=\omega(0)\},\]
where $M^I$ is the space of parametrized paths in  $M$ endowed with the compact-open topology. Let us denote by $\Lambda $ a lifting function belonging to $F_{\bar x}$,   i.e.   
\[\Lambda: \Omega_{F_{\bar x}}\to \A^I\quad  \text{such that}\quad \Lambda(u,\omega)(0)=u\ \text{ and }\  F_{\bar x}\big(\Lambda(u, \omega)(t)\big)=\omega(t),\]
for each $t\in I$  (see \cite{Hur55}), where $\A^I$ is the path space in \A endowed with the compact-open topology.
We have the following lemma  whose proof is an application of \cite[Lemma 2]{DomRab12}. 
\begin{lem}\label{homotopytype}
	The map  $\eta\colon \Omega M \to F^{-1}_{\bar x}(\bar x)$, $\eta(\gamma):=\Lambda \big(0, \gamma\big)(1)$,  is a homotopy equivalence with a homotopy inverse given by the map   $i\circ \big(S(\bar x, \cdot)|_{F^{-1}_{\bar x}(\bar x)}\big)$.  
\end{lem}
\begin{proof}
	By Proposition~\ref{prHurew} and \cite[Lemma 2]{DomRab12}, it is enough to show  that $\A$ is contractible.
	This follows immediately by considering the map
	\baln &H\colon \A\times  [0,1]\to \A, \\	
	&H(u, s)(t):=\begin{cases}
		u(t)&t\in [0,1-s]\\
		0&t\in (1-s,1]
	\end{cases}
	\ealn	
	It is easy to see that $H$ is continuous and it is a homotopy between  the identity of $\A$ and the constant control $0\in \A$.
\end{proof}
Using Lemma~\ref{homotopytype}, we get:
\bt\label{paths}
Let assumptions~\ref{ass1}--\ref{ass2} hold. Then  for all $1\leq p<\infty$ and  any $x,\ y\in M$, the space $\Omega^p_{x,y}$ has the homotopy type of a CW-complex. In particular, the map $S(x,\cdot)|_{\Omega^p_{x,y}}$ is a homotopy equivalence between $\Omega^p_{x,y}$ and $\Ost_{x,y}$.
\et 
\begin{proof}
It is well-known that the inclusion $j\colon \Ost_{\overline{x}, \overline{x}} \to \Omega M$ is a weak homotopy equivalence 
and both the spaces $\Ost_{\overline{x},\overline{x}} $ and $\Omega M$ have the homotopy type of a CW-complex, thus $j$ is also a homotopy equivalence.   From Lemma~\ref{homotopytype} and by the arbitrariness of the base point $\bar x$, we have that  $\Omega^p_{x,x}$ has then the homotopy type of a CW-complex  for every $x\in M$.   Since all the fibers in a Hurewicz fibration with path connected base  have the same homotopy type (see, e.g.,  \cite[Proposition 4.61]{Hatcher}) we have that  $\Omega^p_{x,y}$ has the homotopy type of a CW-complex. Let us consider the following commutative diagram between Hurewicz fibrations
	\[
\begin{tikzcd}
	\A \arrow{r}{F_x}\arrow{d}{S(x,\cdot)}&M\arrow{d}{i_M}\\
	W_x^{1,p}(I,M)\arrow{r}{\tilde F_x}&M
	\end{tikzcd}
\]
where $W_x^{1,p}(I,M)$ is the space of paths in $W^{1,p}(I,M)$ starting at $x$ and $\tilde F_x$ is the endpoint map on $W_x^{1,p}(I,M)$. By the naturality of the long exact sequences of Hurewicz fibrations (see, e.g. \cite[p. 6]{Mitchell}) we then get,  for each $x,y\in M$ and $k\in \N$,  the following commutative diagram: 
\[
	\begin{tikzcd}
\arrow[shorten <=1em]{r}&\pi_{k+1}(M)\arrow{r}\arrow{d}
&\pi_{k}(\Omega^p_{x,y})\arrow{r}\arrow{d}{\big(S(x,\cdot)|_{\Omega^p_{x,y}}\big)_*}&\pi_k(\A)\arrow{r}{(F_x)_*} \arrow{d}{(S(x,\cdot)_*}&[1em]\pi_k(M)\arrow{d}\\
\arrow[shorten <=1em]{r}&\pi_{k+1}(M)\arrow{r}&\pi_{k}(\Ost_{x,y})\arrow{r}&\pi_k(W_x^{1,p}(I,M))\arrow{r}{(\tilde F_x)_*}&[1em]\pi_k(M)
\end{tikzcd}
\]
Since both $\A$ and $W_x^{1,p}(I,M)$ are contractible the above diagram reduces to
\beq\label{diag}
\begin{tikzcd}
	0\arrow{r}&\pi_{k+1}(M)\arrow{r}\arrow{d}
	&\pi_{k}(\Omega^p_{x,y})\arrow{r}\arrow{d}{\big(S(x,\cdot)|_{\Omega^p_{x,y}}\big)_*}&0\arrow{r} &[1em]\pi_k(M)\arrow{d}\\
	0\arrow{r}&\pi_{k+1}(M))\arrow{r}&\pi_{k}(\Ost_{x,y})\arrow{r}&0\arrow{r}&[1em]\pi_k(M)
\end{tikzcd}
\eeq
 which from the five lemma implies that $S(x,\cdot)|_{\Omega^p_{x,y}}$ is a weak homotopy equivalence and then by Whitehead theorem a homotopy  equivalence. 
\end{proof}
From Theorems~\ref{thmHurew} and \ref{paths},  we get the following analogous for the space $\Lambda$ of  \cite[Th. 10]{LerMon19} which concerns the horizontal  free loop  space. 
\bt\label{freeloop}
Let assumptions~\ref{ass1}--\ref{ass2} hold. Then for all $1\leq p<\infty$, $\Lambda^p$ and $\Lst$ are homotopy 
equivalent via  the map $S|_{\Lambda^p}$.
\et
\begin{proof}
	Since $F|_{\Lambda^p}$ is a Hurewicz fibration, after identifying $\triangle_M$ with $M$,  and $F^{-1}(x,x)$ with $\Omega^p_{x,x}$ we 	 have the following long exact sequence 
		\[\ldots\to \pi_k(\Omega^p_{x,x})\rightarrow\pi_k(\Lambda^p)\xrightarrow{(F|_{\Lambda^p})_*}\pi_k(M)\to \pi_{k-1}(\Omega^p_{x,x})\to \ldots\]
	 From   \eqref{diag}, we see that    $\pi_k(\Omega^p_{x,x})$ is isomorphic to $\pi_{k+1}(M)$, for each $k\geq 0$.  Moreover, the map $x\in M\mapsto (x,0)\in \Lambda^p$ is continuous and then $(F|_{\Lambda^p})_*$ is surjective. Thus the  long exact sequence above splits as 
	\[0\to \pi_{k+1}(M)\to \pi_{k}(\Lambda^p)\to \pi_{k}(M)\to 0.\]
	Now, $\Lambda^p$ has the homotopy type of a CW-complex since both $M$ and any fiber $\Omega^p_{x,x}$ do have it as well, thus the fact that the map $S|_{\Lambda ^p}$ is a homotopy equivalence follows, by Whitehead's theorem, if the induced map $(S|_{\Lambda ^p})_*\colon \pi_k(\Lambda^p)\to  \pi_k(\Lst)$ is an isomorphism for each $k\geq 0$. This is a consequence of the following diagram, which is commutative by the naturality   of the long exact sequences of Hurewicz fibrations: 
	\[
	\begin{tikzcd}
		0\arrow{r}&\pi_{k+1}(M)\arrow{r}\arrow{d}{}
		&\pi_{k}(\Lambda^p)\arrow{r}\arrow{d}{(S|_{\Lambda ^p})_*}
		&\pi_k(M)\arrow{d}{}\arrow{r}&0\\
		0\arrow{r}&\pi_{k+1}(M)\arrow{r}&\pi_{k}(\Lst)\arrow{r}&\pi_{k}(M)\arrow{r}&0
	\end{tikzcd}
	\]
\end{proof}



\begin{thebibliography}{10}
	
	\bibitem{AgBaBo20}
	{\sc A.~Agrachev, D.~ Barilari and U. Boscain},
	{\it A Comprehensive Introduction to Sub-{R}iemannian Geometry},
	Cambridge University Press, Cambridge, 2020.
	
	
	
	
	\bibitem{BoaLer17}
	{\sc F.~Boarotto and A.~Lerario}, {\it Homotopy properties of horizontal path 	spaces and a theorem of {S}erre in subriemannian geometry}, Comm. Anal. Geom. 25 
	(2017), pp.~269--301.
	
	
	\bibitem{CaGiMS21}
	{\sc E. Caponio, F. Giannoni, A. Masiello and S. Suhr},
	{\it Connecting and closed geodesics of a {K}ropina metric},
	Adv. Nonlinear Stud. 21 (2021), pp.~683--695.
	
	
	\bibitem{CaJaMa22cor}
	{\sc E.~Caponio, M.~A. Javaloyes, and A.~Masiello}, {\it Multiple connecting geodesics of a   Randers-Kropina metric via homotopy theory for solutions of an affine control system}, Topol. Methods Nonlinear Anal. 61 (2023), pp. 527--547; with a corrigendum, in press, available at \href{https://doi.org/10.48550/arXiv.2409.19596}{arXiv:2409.19596 [math.DG]}.
	
	
	
	\bibitem{Clarke90}
	{\sc F. H. Clarke}, {\it Optimization and Nonsmooth Analysis}, Society for Industrial and Applied Mathematics (SIAM), Philadelphia, PA, 1990.
	
	
	\bibitem{DomRab12}
	{\sc J.~Dominy and H.~Rabitz}, {\it Dynamic homotopy and landscape dynamical
		set topology in quantum control}, J. Math. Phys. 53 (2012),
	082201.
	
	
	
	\bibitem{Hartma64}
	{\sc P. Hartman}, {\it Ordinary Differential Equations}, {John Wiley \& Sons Inc., New York}, 1964.
	
	\bibitem{Hatcher}
	{\sc A. Hatcher}, {\it Algebraic Geometry}, Cambridge University Press, Cambridge, 2001.
	
	\bibitem{Hur55}
	{\sc W. Hurewicz}, 
	{\it On the concept of fiber space}, 
	{Proc. Nat. Acad. Sci. U. S. A.} 
	41 (1955), pp.~ 956--961.
	
	
	\bibitem{LerMon19}
	{\sc A. Lerario and A. Mondino},
	{\it Homotopy properties of horizontal loop spaces and applications to closed sub-{R}iemannian geodesics},
	{Trans. Amer. Math. Soc.  Ser. B}, 6 (2019), pp. 187--214.
	
	\bibitem{Mitchell}
	{\sc S. A. Mitchell}, {\it Notes on Serre Fibrations}, available at \href{https://sites.math.washington.edu/~mitchell/Atopc/serre.pdf}{https://sites.math.washington.edu/~mitchell/Atopc/serre.pdf}, 2001.
		
	\bibitem{Muller} 
	{\sc O.~M{\"u}ller}, {\it A note on closed isometric embeddings}, J. Math. Anal. Appl., 349 (2009), 297--298.
	
	\bibitem{Rudin}
	{\sc W. Rudin}, {\it Real and Complex Analysis}, McGraw-Hill, Singapore,  1987.
		
	\bibitem{Sarichev}
	{\sc A.V.~ Sarichev}, 
	{\it On homotopy properties of the space of trajectories of a completely nonholonomic differential system}, 
	Soviet. Math. Dokl., 42 (1991), pp. 674-- 678.
	
	
	\bibitem{Serre}
		{\sc J-P. Serre}, {\it Lie Algebras and {Lie} Groups}, Springer-Verlag, Berlin, Heidelberg, New York, 2006.
		
	\bibitem{Smale58} {\sc S. Smale},  {\it Regular curves on Riemannian manifolds}, Trans. Amer.
	Math. Soc., 87 (1958),  492--512.
	
	
\end{thebibliography}
\end{document}